\documentclass[12pt]{amsart}
\usepackage{amsfonts}
\usepackage{amsmath}
\usepackage{amssymb}
\usepackage{mathrsfs}
\usepackage{graphicx}
\usepackage{enumerate}
\usepackage{multicol}
\usepackage{a4wide}
\usepackage[all,cmtip]{xy}

\newtheorem{theorem}{Theorem}
\newtheorem{corollary}{Corollary}

\newtheorem{definition}{Definition}
\newtheorem{proposition}{Proposition}
\theoremstyle{remark}
\newtheorem{remark}{Remark}

\DeclareMathOperator{\id}{\it id}
\DeclareMathOperator{\pr}{pr}

\DeclareMathOperator{\ad}{ad}
\DeclareMathOperator{\Ad}{Ad}

\DeclareMathOperator{\SU}{SU}

\DeclareMathOperator{\Diff}{Diff}
\DeclareMathOperator{\Vect}{Vect}
\DeclareMathOperator{\Vir}{Vir}
\DeclareMathOperator{\Mob}{\text{M\"ob}}

\DeclareMathOperator{\Lieg}{\mathfrak{g}}
\DeclareMathOperator{\Lieh}{\mathfrak{h}}
\DeclareMathOperator{\Liep}{\mathfrak{p}}
\DeclareMathOperator{\Liek}{\mathfrak{k}}
\DeclareMathOperator{\Lied}{\mathfrak{d}}

\DeclareMathOperator{\comp}{{\mathbb{C}}}
\DeclareMathOperator{\integer}{{\mathbb{Z}}}
\DeclareMathOperator{\real}{\mathbb{R}}
\DeclareMathOperator{\unitD}{\mathbb{D}}

\DeclareMathOperator{\calF}{\mathcal{F}}
\DeclareMathOperator{\calG}{\mathcal{G}}
\DeclareMathOperator{\calH}{\mathcal{H}}
\DeclareMathOperator{\calA}{\mathcal{A}}

\DeclareMathOperator{\calE}{\mathcal{E}}

\DeclareMathOperator{\calD}{\mathcal{D}}

\DeclareMathOperator{\frakA}{\mathfrak{A}}

\DeclareMathOperator{\mob}{\mathfrak{m}\ddot{\mathfrak{o}}\mathfrak{b}}

\DeclareMathOperator{\nuk}{ker}

\DeclareMathOperator{\Rot}{Rot}
\DeclareMathOperator{\rot}{\mathfrak{rot}}

\DeclareMathOperator{\wDiff}{\widetilde{Diff}}
\DeclareMathOperator{\llangle}{\langle \! \langle}
\DeclareMathOperator{\rrangle}{\rangle \! \rangle}

\title[Sub-Riemannian structures...]{Sub-Riemannian structures corresponding to K\"ahlerian metrics on the universal Teichm\"uller space and curve}
\author[E. Grong, I. Markina, A. Vasil'ev]{Erlend Grong, Irina Markina, and Alexander Vasil'ev}

\address{Department of Mathematics,
University of Bergen, P.O.~Box~7803, Bergen N-5020, Norway}
\email{erlend.grong@math.uib.no}
\email{irina.markina@math.uib.no}
\email{alexander.vasiliev@math.uib.no}

\thanks{The authors have been  supported by the grants of the Norwegian Research Council \#204726/V30 and  \#213440/BG}


\begin{document}

\begin{abstract}
We consider the group of sense-preserving diffeomorphisms $\Diff S^1$ of the unit circle and its central extension, the Virasoro-Bott group, with their respective horizontal distributions chosen to be Ehresmann connections with respect to a projection to the smooth universal Teichm\"uller space and the universal Teichm\"uller curve associated to the space of normalized univalent functions. We find formulas for the normal geodesics with respect to the pullback of the invariant K\"ahlerian metrics, namely, the Velling-Kirillov metric on the class of normalized univalent functions and  the Weil-Petersson metric on the universal Teichm\"uller space. The geodesic  equations are sub-Riemannian analogues of the Euler-Arnold equation and lead  to the CLM, KdV, and other known non-linear PDE.
\end{abstract}
\maketitle

\section{Introduction}

Arnold \cite{Arnold} proposed in 1966 a program of the geometric approach to hydrodynamics, which ultimately led  to  general geodesic equations on Lie algebras
of infinite-dimensional Lie-Fr\'echet groups of volumemorphisms of a compact finite-dimensional Riemannian manifold, arriving at certain known equations of mathematical physics regarding a chosen Riemannian metric. These equations are now often referred to as the Euler-Arnold equations, see also~\cite{K}.
The Lie-Fr\'echet group of sense-preserving diffeomorphisms $\Diff S^1$ of the unit circle $S^1$ is one of the simplest, and then, important examples of infinite-dimensional Lie groups modeled on the Fr\'echet space of all $C^{\infty}$-smooth functions $h\colon S^1=\mathbb R / 2\pi \mathbb Z\to\mathbb R$ endowed with the countable family of seminorms $\|h\|_n=\max_{\theta\in S^1}|\frac{d^n}{d\theta^n}h(\theta)|$, $n\geq 0$. The interest to this group comes from Conformal Field Theory, where this group together with its universal central extension, the Virasoro-Bott group $\Vir$, occurs as a space of reparametrizations of a closed string. Two non-trivial (of three possible) coadjoint orbits of the group $\Vir$ are the homogeneous spaces $B=\Diff S^1/\Rot$ and $M=\Diff S^1/\Mob$, where $B$ is a holomorphic disk fiber space over $M$,   $\Mob$ is the group of M\"obius automorphisms of the unit disk restricted to $S^1$, and $\Rot\simeq S^1$ is its subgroup of rotations associated to the circle $S^1$ itself, see e.g., \cite{K, Kirillov202, Segal}. The space $M$ is referred to as a smooth approximation of the {\it universal Teichm\"uller space} $\mathcal T$, see~\cite{TT}, and $B$ as a smooth approximation of the {\it universal Teichm\"uller curve} $\mathcal T(1)$.  Moreover, the natural inclusion $M\hookrightarrow \mathcal T(1)$ is holomorphic~\cite{NV}. The space $B$ contains all necessary information on the construction of the unitary representation of $\Diff S^1$ due to Kirillov and Yur'ev \cite{Kirillov4, Unirreps}. The group $\Diff S^1$ acts on $B$ and $M$, and it is natural to consider the manifold $B$ as a base space for the principal bundle $\Rot{\longrightarrow} \Diff S^1\stackrel{\pi_1}{\longrightarrow} B$, and the manifold $M$ as a base space for the principal bundle $\Mob {\longrightarrow} \Diff S^1\stackrel{\pi_2}{\longrightarrow} M$. The Lie algebra of $\Diff S^1$ is identified with the space $\Vect S^1$ of all smooth real vector fields on $S^1$ with the Lie brackets as the negative of the usual commutator. This identification can be made by associating the equivalence class of curves $[t\to\gamma(t)]\in T_1\Diff S^1$ with the vector field on $S^1$, $Xh(\theta)=\frac{d}{dt}h(\gamma(t))\big|_{t=0}$, $\gamma(0)=\theta$, where $h\in C^{\infty}(S^1, \mathbb R)$ and $\theta\in S^1$.  We write $v\in \Vect S^1$ instead of $v\partial_{\theta}$.  Let us define a real valued form
$\eta_0$,   associating to every $v\in \Vect S^1$ its mean value
\[
\eta_0(v)=\frac{1}{2\pi}\int_{0}^{2\pi}v(\theta)d\theta,
\]
and
the complex valued form
$\eta_1$,   associating to every $v\in \Vect S^1$ the number
\[
\eta_1(v)=\frac{1}{2\pi}\int_{0}^{2\pi}e^{-i\theta}v(\theta)d\theta.
\]
Let us denote by $\Vect_0S^1$ the kernel $\nuk\eta_0$ and let $\Lied = \ker \eta_0 \cap \ker \eta_1$ denote the complement to the Lie algebra $\mob$ of the group $\Mob$ in $\Vect S^1$.

Then we are able to define subbundles $\mathcal H$ and $\mathcal D$ of $T\Diff S^1$  by left translations of $\Vect_0S^1$ and $\Lied$ by $\Diff S^1$ respectively. 
Similarly, we define  subbundles $\mathcal E$ and $\mathcal C$ of the central extension $\Vir$ of $\Diff S^1$ obtained by left translations of $(\Vect_0S^1,0)$ and $(\Lied,0)$ by $\Vir$. 

Let $\mathbf{g}$ be a Riemannian  metric on $T\Diff S^1$, and let  $\mathbf{h}_{\mathcal H}$ be its restriction to $\mathcal H$ and let $\mathbf{h}_{\mathcal D}$ be its restriction to $\mathcal D$. Correspondingly, if $\mathbf{u}$ is a Riemannian  metric on $\Vir$, then we denote by  $\mathbf{h}_{\mathcal E}$ and  $\mathbf{h}_{\mathcal C}$ its restriction to $\mathcal E$ and $\mathcal C$.

Denote by $\rot$ the subalgebra of $\Vect S^1$ of constant vector fields corresponding to the subgroup of rotations $S^1$. Then $\Vect S^1= \Vect_0 S^1\oplus \rot=\Lied\oplus \mob$ and $T\Diff S^1=\mathcal H\oplus \mathcal R=\mathcal D\oplus \mathcal M$, where $\mathcal R$ and $\mathcal M$ are subbundles obtained by left translations of $\rot$ and $\mob$. Notice that $\mathcal R=\ker d\pi_1$ and  $\mathcal M=\ker d\pi_2$. Therefore, the subbundles $\mathcal H$ and $\mathcal D$ of $T\Diff S^1$ are the Ehresmann connections. Similarly, the subbundles $\mathcal E$ and $\mathcal C$ are the Ehresmann
connections on $\Vir$.

A smooth curve $\gamma \colon [0,1]\to\Diff S^1$  is called $\mathcal H$- horizontal  if $\dot \gamma\in \mathcal H_{\gamma(t)}$  for every $t\in[0,1]$. Similar definitions are valid for the distributions $\mathcal D$, $\mathcal E$, and $\mathcal C$.

We look for  $\mathcal H$-horizontal curves $\gamma(t)$ connecting two points $a_0$ and $a_1$, $\gamma(0)=a_0$, $\gamma(1)=a_1$ on $\Diff S^1$, that give the critical values for the energy functional 
\begin{equation}\label{energy} \nonumber
E(\gamma)=\frac{1}{2}\int_0^1\mathbf{h}(\dot{\gamma}, \dot{\gamma})dt.
\end{equation}
Analogously, we formulate the problem for the groups $\Diff S^1$ and $\Vir$ and the distributions $\mathcal D$, $\mathcal E$, and $\mathcal C$. The objects $(\Diff S^1, \mathcal H, \mathbf{h}_{\mathcal H})$, $(\Diff S^1, \mathcal D, \mathbf{h}_{\mathcal D})$, $(\Vir, \mathcal E, \mathbf{h}_{\mathcal E})$ and $(\Vir, \mathcal C, \mathbf{h}_{\mathcal C})$ are
infinite-dimensional analogs to the classical sub-Riemannian manifolds which in finite dimensions have been actively studied recently, and widely documented, see e.g., \cite{AS, Mon, Str1, Str2}.
In the present paper we address two problems. The first one, the problem of controllability, or whether it is possible to join arbitrary points $a_0$ and $a_1$ on $\Diff S^1$ or $\Vir$
by $\mathcal H$- (or $\mathcal E$-) horizontal curves, was treated in \cite{GMV}.  The second one is to find geodesic equations for critical curves with respect to the metric $\mathbf{h}$ with the corresponding index $\mathcal H$, $\mathcal D$, $\mathcal E$, or $\mathcal C$. These equations are sub-Riemannian analogues of the Euler-Arnold equation. The metrics are chosen to be either Sobolev, or  the pullback of the invariant K\"ahlerian metric, in particular the {\it Velling-Kirillov metric}~\cite{Kirillov1, Kirillov4, Vel}, on the class of normalized univalent functions related to $B\simeq \mathcal T(1)$ by conformal welding, or  with respect to the pullback of the {\it Weil-Petersson metric} on the universal Teichm\"uller space~$\mathcal T$. We find that the geodesic equations are analogues of the Constantin-Lax-Majda (CLM), Camassa-Holm, Huter-Saxton, KdV, and other known non-linear PDE. Inspired by the above problems  we develop an analogue of sub-Riemannian geometry on infinite-dimensional Lie groups. Equations for the sub-Riemannian geodesics for $\mathbf h_{\calH}$ and $\mathbf h_{\calE}$ previously appeared in~\cite{GMV}.

 \section{Infinite-dimensional Lie groups with constraints}\label{Lie group}

In this section we apply variational calculus to determine sub-Riemannian geodesics for infinite-dimensional Lie groups with invariant subbundles. This is a special case of the calculus developed by the authors in \cite{GMV} for finding geodesics in general infinite-dimensional manifolds. In particular, we introduce {\it semi-rigid curves} that play a similar role
to abnormal geodesics in finite-dimensional sub-Riemannian geometry. We will work with Lie groups modeled on  {convenient vector spaces} following the terminology  found in \cite{KrieglMichor}. A convenient vector space is a locally convex vector space, where the most general notion of smoothness, based on the notion of smooth curves, is introduced and the vector space satisfies a weak completeness condition which is called $c^{\infty}$-completeness. The respective topology is given by  $c^{\infty}$-open sets of a convenient vector space. For a short introduction, we refer the reader to \cite{Michor} or \cite{KMGroup}. In particular, Fr\'echet
spaces are convenient.

\subsection{Regular Lie groups}
Let $G$ be a Lie group modeled on $c^{\infty}$-open sets of a convienient vector space with the Lie algebra $\Lieg$.  We use the symbol $\ell_a$ to denote the left multiplication by an element $a\in G$. Let us define the {\it left Maurer-Cartan form} $\kappa^\ell$ as a $\Lieg$-valued one-form on $G$, given by the formula
$$\kappa^{\ell}(v) = d \ell_{a^{-1}} v, \qquad v \in T_aG.$$
Let us use the notation $C^{\infty}(\real,G)$ for the space of smooth maps $\gamma\colon \real\to G$, and the notation  $C^{\infty}(\real,\mathfrak g)$ for the convenient vector space of smooth maps from $\real$ to the Lie algebra $\mathfrak g$. 
To any smooth curve $\gamma\colon \real\to G$ one associates a smooth curve $u(t) = \kappa^\ell(\dot \gamma(t))$, $t\in \real$ in the Lie algebra $\Lieg$ which is called {\it the left logarithmic derivative} of $\gamma$. 
Throughout the paper we  assume  that  the Lie groups are {\it regular}, which essentially requires that the above correspondence from $\gamma\in G$ to $u\in\mathfrak g$ remains true the other way around. Let us give a precise definition.

\begin{definition}[\cite{KMGroup,Milnor}]
A Lie group $G$ is called regular if it satisfies the following two properties.
\begin{itemize}
\item[1.] Any smooth curve 
$$
\begin{array}{ccccc}
u\colon & \real & \to & \Lieg
\\
& t  & \mapsto & u(t)
\end{array}
$$ 
is the left logarithmic derivative of a curve $\gamma\in C^{\infty}(\real,G)$ with $\gamma(0) = \mathbf 1$, where $\mathbf1$ stands for the identity of the group $G$;
\item[2.] The mapping
$$
\begin{array}{ccc} 
C^\infty(\real, \Lieg) & \to & G 
\\ 
{[t \mapsto u(t)]} & \mapsto  & \gamma(1) 
\end{array}
$$
is smooth. Here $\gamma$ is a solution to the equation $\kappa^\ell(\dot \gamma(t)) = u(t)$, $t\in \real$ with the initial data $\gamma(0) = \mathbf 1$.
\end{itemize}
\end{definition}

Let us make the following remarks on regular Lie groups. 

$\bullet\ $
There are no known examples of non-regular Lie groups. The term `regular' is also used for somewhat stricter conditions, see~\cite{LieFre}. 

$\bullet\ $ The second condition of the above definition of regularity is a generalization of the exponential map produced by the constant map $[t\to u_0]\in C^\infty(\real, \Lieg)$. Thus, if a Lie group is regular, then the exponential map $\exp_G: \Lieg \to G$ exists and it is smooth. However, many of the properties that we are used for the group exponential map in finite dimensions do not necessarily hold in infinite dimensions. For example, it can happen that the exponential map is not locally surjective and  the  Baker-Campbell-Hausdorff formula does not work.

$\bullet\ $ If there is a curve $\gamma$, starting from $\mathbf 1\in G$ with the left logarithmic derivative $u(t)$, then for any $a \in G$, there is a curve $\widetilde \gamma$ starting from $a$ and having the same left logarithmic derivative $u(t)$.

$\bullet\ $ Regularity of a Lie group can be similarly defined in terms of the right logarithmic derivative. Let $r_a$ denote the right translation by $a$, and let $\kappa^r(v) = d r_{a^{-1}} v$, $v \in T_aG$ be the right Maurer-Cartan form.  Then for a given $\gamma\colon \real\to G$, the curve $u(t) = \kappa^r(\dot \gamma(t))$, $t\in \real$, is called the right logarithmic derivative. In this case regularity of the group implies uniqueness of the solution to the initial value problem $\kappa^r(\dot \gamma(t)) = u(t), \gamma(0) = \mathbf 1$. The property of a group to be regular does not depend on the choice between left or right translations in the definition. 

\subsection{Variational calculus on regular Lie groups} \label{sec:VarCalc}
From now on, we parametrize all curves on the domain $I := [0,1]$ unless otherwise is stated.

Let $G$ be a regular Lie group modeled on  $c^\infty$-open subsets of a convenient vector space. We say that a smooth map 
\begin{equation}\label{gammas}
\begin{array}{ccccc}
F\colon I\times(-\epsilon,\epsilon) & \to &G 
\\  
(t,s) &\mapsto & F(t,s)
\end{array}
\end{equation}
is a {\it variation of a curve $\gamma\colon I\to G$} if 
\begin{equation}\label{fixpoints}
F(t,0)=\gamma(t),\quad F(0,s)=\gamma(0),\quad F(1,s)=\gamma(1).
\end{equation}
We will write a variation simply as $\gamma^s$ rather than $F$ for the sake of simplicity, where $\gamma^s(t)=F(t,s)$. Observe that $[s \mapsto \gamma^s]$ can be considered as a smooth curve from $(-\epsilon,\epsilon)$ to the space $C^\infty(I, G)$ with $\gamma^0 = \gamma$. 

Let $G$ be a regular Lie group whose Lie algebra $\Lieg$ is endowed with an inner product  $\langle  \cdot ,\cdot \rangle$. Let $\mathbf g$ be a left-invariant metric on $G$ obtained by left translation of this inner product,~i.~e., 
$$\mathbf g(v,w) = \langle \kappa^\ell(v), \kappa^\ell(w) \rangle,\qquad  v,w\in T_aG,\quad \text{for $\forall \ a\in G$}.$$
Define the energy functional as $E(\gamma) =\frac{1}{2} \int_0^1 \mathbf g(\dot \gamma, \dot \gamma) \, dt$. We want to describe the  curves which are the critical points of the energy functional and satisfy $\gamma(0) = a_0$ and $\gamma(1) = a_1$, for two given points $a_0, a_1 \in G$,~i.~e.,   such curves $\gamma$ that satisfy the equation
$$\partial_s E(\gamma^s)|_{s=0} = 0, \text{ for any variation } \gamma^s.$$
We call them  Riemannian geodesics.

In order to write down the geodesic equations, we give the following observations. For a variation $\gamma^s$ of $\gamma$, we define curves $u^s$ and $z$ in $\Lieg$ by $u^s(t)= \kappa^\ell(\dot\gamma^s(t))$ and
\begin{equation} \label{eq:firstvariation} 
z(t) = \kappa^\ell(\partial_s \gamma^s(t)) |_{s=0}. 
\end{equation}
They are related by the known equality  
\begin{equation}\label{eq:partials}
\partial_s u^s(t)|_{s=0} = \dot z(t) + [u(t), z(t)].
\end{equation}
Indeed, by making use of the Cartan equation $d\kappa^\ell(v,w) = - [\kappa^\ell(v), \kappa^\ell(w)]$, we get
\begin{equation}\label{eq:1}
d\kappa^\ell(\partial_s \gamma^{s}(t), \partial_t\gamma^{s}(t))|_{s=0} =-[\kappa^\ell(\partial_s\gamma^{s}(t)),\kappa^\ell(\partial_t\gamma^{s}(t))]|_{s=0}= [u(t), z(t)]. \nonumber
\end{equation}
On the other hand, if $F^*$ denotes the pullback by $F$ in~\eqref{gammas}, then we obtain
\begin{align}\label{eq:2}
d\kappa^\ell(\partial_s \gamma^{s}(t), \partial_t\gamma^{s}(t)))|_{s=0} & = d (F^*\kappa^\ell)(\partial_s, \partial_t)|_{s=0} \nonumber
\\
&=\partial_s \big((F^*\kappa^\ell)(\partial_t)\big)|_{s=0} - \partial_t\big( (F^*\kappa^\ell)(\partial_s)\big)|_{s=0} \nonumber
\\
& = \partial_s\kappa^\ell(\partial_t\gamma^s(t))|_{s=0}-\partial_t \kappa^\ell(\partial_s\gamma^s(t)))|_{s=0} = \partial_s u^{s}(t)|_{s=0} - \partial_{t} z(t). \nonumber
\end{align}
Inspired by \eqref{eq:partials}, we introduce the linear map for any fixed $u\in C^{\infty}(I,\mathfrak g)$
\begin{equation}\label{tu}
\begin{array}{rccc} \tau_u: & C^\infty(I, \Lieg) & \to & C^\infty(I, \Lieg) \\
& x & \mapsto & \dot x + [u,x]  \end{array}.
\end{equation}
This allows us to rewrite \eqref{eq:partials} as $\partial_s u^s|_{s=0} = \tau_u z$. We remark the following.

 \begin{proposition}[\cite{GMV}] \label{lemma:insert0}
For any $y$ and $u \in C^\infty(I,\Lieg)$, there is a unique $x \in C^\infty(I,\Lieg)$, satisfying
\begin{equation} \label{tauinverse} \tau_u x = y, \qquad x(0) = 0.\end{equation}
\end{proposition}
If $x$ satisfies \eqref{tauinverse}, then we  write $x = \tau_u^{-1} y.$ Explicitly,
$$
\tau_u^{-1}(y)(t)= \Ad_{\gamma(t)^{-1}} \int_0^t \Ad_{\gamma(\tilde t)} y(\tilde t) \, d\tilde t,
$$
where $\Ad$ is the adjoint action of $G$ on $\mathfrak g$, and $\gamma$ is a curve with the left logarithmic derivative $u$.

We define an inner product on the space $C^\infty(I, \Lieg)$ by
$$\llangle x, y \rrangle = \int_0^1 \langle x(t), y(t) \rangle \, dt,\qquad x,y\colon I\to \Lieg.$$
Then  the variation of the energy functional $E=\frac{1}{2} \int_0^1 \mathbf g(\dot \gamma(t) , \dot \gamma(t)) \, dt$ is written as
\begin{equation}\label{product}
 \partial_s E(\gamma^s) |_{s=0}  = \llangle u, \tau_u z \rrangle,
\end{equation}
for any variation $\gamma^s$ of $\gamma$ and $z$ defined by~\eqref{eq:firstvariation}. 
Indeed, the equation~\eqref{eq:partials} and the definition of the map $\tau_u$ imply
\begin{align*} \partial_s E(\gamma^s) |_{s=0} & = \int_0^1 \langle u(t), \partial_s u^s(t) |_{s=0} \rangle dt \\
& = \int_0^1 \langle u(t), \dot z(t) + [u(t), z(t)] \rangle dt =  \int_0^1 \langle u(t), \tau_u(z)(t) \rangle dt 
=\llangle u, \tau_u z \rrangle.
\end{align*}

Let $\frakA$ be the collection of curves
\begin{equation}\label{AA}
\frakA = \{x \in C^\infty(I,\Lieg) \, : \, x(0) = x(1) = 0\},
\end{equation}
and let $\tau_u\frakA$ be its image under the map $\tau_u$.
It is obvious that $z$ defined in~\eqref{eq:firstvariation} belongs to $\frakA$. In order to study the critical points of the energy functional $E$, we first  characterize the orthogonal complement $(\tau_u \frakA)^\perp$ to $\tau_u\mathfrak A$ with respect to the inner product $\llangle \cdot , \cdot \rrangle$. We assume that the adjoint map $\ad_x^{\top}$ to  $\ad_x: y \mapsto [x,y]$ exists for any $x \in \Lieg$ with respect to the inner product $\langle \cdot , \cdot \rangle$. 

\begin{proposition}\label{orthog}
Let $\gamma$ be a curve in $G$ with the left logarithmic derivative $u$. 
If $w \in (\tau_u \frakA)^\perp$, then $w$ is a solution to the equation 
\begin{equation}\label{diffeq}
\dot w = \ad_u^{\top} (w).
\end{equation}
\end{proposition}
\begin{proof}
If $w \in (\tau_u \frakA)^\perp$, then for any $x \in \frakA$, we have
\begin{align*}
0 = \llangle w, \tau_u x \rrangle & = \int_0^1 \langle w , \dot x + [u,x] \rangle \, dt = - \int_0^1 \langle \dot w - \ad_u^{\top}(w) , x \rangle \, dt
\end{align*}
by integration by parts. Hence $w$ is a solution to $\dot w = \ad_u^{\top}(w).$
\end{proof}
The equation \eqref{diffeq} is the left Euler-Arnold equation on $G$.

\subsection{Horizontal geodesics} \label{sec:GroupCritical}
In this section we define the left-invariant sub-Riemannian structure on a Lie group and study the set of critical points of the energy functional defined by a sub-Riemannian metric.

Let $\mathbf g$ be a left-invariant metric on $G$ corresponding to an inner product $\langle \cdot , \cdot \rangle$ in the Lie algebra $\mathfrak g$.
Choose a $c^\infty$-closed subspace $\Lieh$ of $\Lieg$, such that $\Lieh \oplus \Lieh^\perp = \Lieg$, where $\Lieh^\perp$ is  orthogonal to $\Lieh$ with respect to $\langle\cdot , \cdot \rangle$. Define a smooth subbundle $\calH$ of $TG$ by  left translations of $\Lieh$, or equivalently, the subbundle of all vectors $v$ with $\kappa^\ell(v) \in \Lieh$. Denote by $\mathbf h$ the restriction of the metric $\mathbf g$ to the subbundle $\calH$. We call the pair $(\calH,\mathbf h)$ the {\it left-invariant sub-Riemannian structure} on the Lie group $G$. 

 We say that a smooth curve $\gamma:I \to G$ is {\it horizontal} if $\dot \gamma(t) \in \calH_{\gamma(t)}$ for any $t \in I$. Similarly, a variation $\gamma^s$ of a curve $\gamma$ defined in (\ref{gammas}--\ref{fixpoints}) is called a horizontal variation if $\partial_t \gamma^s(t)\in\calH_{\gamma^s(t)}$ for all $s\in(-\epsilon,\epsilon)$ and $t\in I$. We want to describe the horizontal curves connecting two given fixed points which are the critical points for the energy functional, defined on the space of horizontal curves
$$E(\gamma) = \frac{1}{2}\int_0^1 \mathbf h(\dot \gamma, \dot \gamma) \, dt$$
Similarly to the Riemannian case, we introduce the following definition.

\begin{definition}\label{critical}
A horizontal curve $\gamma$ is called a sub-Riemannian geodesic if 
$$
\partial_s E(\gamma^s) |_{s=0} = 0\ \ \text{for any horizontal variation}\ \ \gamma^s.
$$
\end{definition}  
The collection of all horizontal variations of a curve $\gamma$ is denoted by $\mathcal J_{\calH}(\gamma)$. Let us introduce the notation
\begin{equation}\label{b}
\mathfrak{Var}_{\calH}(\gamma)=\{z\in \frakA\ \mid\ \text{there is}\ \gamma^s\in\mathcal J_{\calH}(\gamma)\ \text{such that}\ z=\kappa^\ell(\partial_s \gamma^s |_{s=0})\}.
\end{equation}
By the discussion in Section~\ref{sec:VarCalc}, we know that $\gamma$ is a sub-Riemannian geodesic if and only if its left logarithmic derivative $u$ satisfies $\llangle u, \tau_u z \rrangle = 0$ for any $z \in \mathfrak{Var}_{\calH}(\gamma)$. 

Define the subset $$\frakA_{\calH}=\tau^{-1}_u\pr_{\Lieh}\tau_u\frakA$$
of $\frakA$, where $\pr_{\Lieh}:\Lieg \to \Lieh$ is the orthogonal projection. Obviously, the inclusion $\mathfrak{Var}_{\calH}(\gamma)\subseteq\frakA_{\calH}$ holds.
Indeed, let $z\in\mathfrak{Var}_{\calH}(\gamma)$ and the curve 
$\gamma^s\in\mathcal J_{\calH}(\gamma)$ be such that $z=\kappa^\ell(\partial_s \gamma^s |_{s=0})$. The left logarithmic derivative $u^s$ of $\gamma^s$ is in $\mathfrak h$, and $s \mapsto u^s$ is a smooth curve from $(-\epsilon,\epsilon)$ to $C^\infty(I, \mathfrak h)$. Thus, $\tau_u z=\partial_s u^s|_{s=0}\in (\tau_u \frakA) \cap\, C^{\infty}(I,\mathfrak h)= \tau_u \frakA_{\calH}$. However, it is not necessarily true that $\mathfrak{Var}_{\calH}(\gamma) = \frakA_{\calH}.$ This phenomenon appears for both finite and infinite dimensions, and was observed long time ago, see~\cite{BH,Hamenstadt}. We will use the term {\it semi-rigid} for curves for which this property fails.

\begin{definition}\label{semi-rigid}
A horizontal curve $\gamma$ is called semi-rigid if  $\mathfrak{Var}_{\calH}$ is a proper subspace of~$\mathfrak{A}_{\calH}$. \end{definition}

Since geodesics are curves $\gamma$ with left logarithmic derivative $u\in(\tau_u\mathfrak{Var}_{\calH})^{\bot}$,  and \\
$(\tau_u\mathfrak{Var}_{\calH})^{\perp}~\supseteq~(\tau_u\mathfrak A_{\calH})^{\perp}$ by the inclusion $\tau_u\mathfrak{Var}_{\calH}\subseteq\tau_u\mathfrak A_{\calH}$, we know that $u \in (\tau_u\mathfrak A_{\calH})^{\perp}$ is a sufficient condition for $\gamma$ to be a geodesic.

\begin{definition}\label{norcr}
A horizontal curve $\gamma$ whose left logarithmic derivative $u$ belongs to $\big(\tau_u\mathfrak A_{\calH}\big)^{\perp}$ is called a {\it normal sub-Riemannian geodesic}.
\end{definition}

Notice that if a curve is not semi-rigid, then it is a sub-Riemannian geodesic if and only if it is a normal sub-Riemannian geodesic. Since the definitions of geodesics and normal geodesics use the    
 orthogonal complement to $\tau_u\mathfrak{Var}_{\calH}(\gamma)$, they essentially depend on the choice of a metric. The definition of semi-rigid curves does not depend on the metric but rather on the properties of the horizontal distribution $\calH$ itself.
We emphasize that, according to the definitions, a curve $\gamma$ can be both semi-rigid and normal geodesic at the same time. Semi-rigid curves need not be geodesics, but all geodesics which are not normal are semi-rigid. We summarize the results of the section in the following statement.

\begin{theorem} \label{theorem:LieGroups}
Let $G$ be a regular Lie group  with the Lie algebra $\Lieg$. Assume that
$\Lieg$ is equipped with an inner product $\langle \cdot , \cdot \rangle$, the adjoint map $\ad^{\top}_x$ to $\ad_x$ is well-defined, and $\Lieg = \Lieh \oplus \Lieh^\perp$, where $\Lieh^\perp$ is the orthogonal complement to $\Lieh$ with respect to $\langle \cdot , \cdot \rangle$. Let $(\calH,\mathbf h)$ be the corresponding left-invariant sub-Riemannian structure on $G$.
If a horizontal curve $\gamma$ is a geodesic, then it is either a semi-rigid curve or it is a normal geodesic. In the latter case it is a solution to the equations
$$u = \kappa^\ell(\dot \gamma),\qquad\dot u = \pr_{\Lieh} \ad_u^{\top} (u+ \lambda), \quad \dot \lambda = \pr_{\Lieh^\perp} \ad_u^{\top} (u+ \lambda),$$
for some curve $\lambda$ in $\Lieh^\perp$. 
\end{theorem}

We can repeat the above statements for the right-invariant sub-Riemannian structure. The right Maurer-Cartan form satisfies the equation $d\kappa^r(v,w) = [\kappa^r(v) , \kappa^r(w)]$ that implies the new definition  $\tau_u\colon x \to \dot x - [u, x]$. The map $\tau_u$ is also invertible and the equations for the normal geodesics become
$$u = \kappa^r(\dot \gamma),\qquad
\dot u = -\pr_{\Lieh} \ad_u^{\top} (u+ \lambda), \quad \dot \lambda = -\pr_{\Lieh^\perp} \ad_u^{\top} (u+ \lambda).
$$

\subsection{Semi-rigid curves on regular Lie groups}
In Section \ref{sec:GroupCritical} we defined  semi-rigid curves in terms of horizontal variations of a curve $\gamma$. But we could have also described them purely in terms of its left logarithmic derivative $u$.

Let $\Lieg = \Lieh \oplus \Liek$ be a splitting of $\Lieg$. The subspace $\Liek$ here is the topological complement to  $\Lieh$ in $\Lieg$. We do not need to introduce a metric, since semi-rigid curves do not depend on it. Define a subbundle $\calH$ of $TG$  generated by left translations of $\Lieh$. Then a results found in \cite[Lemma 8.8]{Milnor} allows us to describe variations only in terms of the Lie algebra.

Denote by $\pr_{\Liek}\colon\Lieg\to\Liek$ the projection with the kernel  $\Lieh$. 
\begin{proposition}[\cite{GMV}]
Let  $\gamma\colon I\to G$ be a horizontal curve with the left logarithmic derivative $u\colon I\to\Lieg$. A curve $\gamma$ is semi-rigid, if and only if, there is a curve $z \in \frakA$ with
\begin{equation} \label{eq:zHvar} \pr_{\Liek} \tau_u z = 0,\end{equation}
such that the problem
\begin{equation} \label{eq:Problem}
\begin{cases}
\partial_s u^s = \tau_{(u^s)} z^s, 
\\
u^s(t) \in \Lieh, \qquad\text{for}\quad (t,s)\in I \times (-\epsilon, \epsilon),
\\
z^s(t) \in \Lieg, \qquad\text{for}\quad (t,s)\in I \times (-\epsilon, \epsilon),
\\
u^0(t) = u(t), \quad z^0(t) = z(t), \qquad\text{for}\quad t\in I,
\\
z^s(0) = z^s(1) = 0, \qquad\text{for}\quad s\in (-\epsilon, \epsilon),
\end{cases} 
\end{equation}
has no solution.
\end{proposition}

\subsection{Geodesics with respect to invariant metrics} 

Now we consider a special situation when a Lie group $G$ carries a metric  invariant under the action of some special subgroup $K$ of $G$. Namely, let $G$ be a finite- or infinite-dimensional regular Lie group and $K$ be a connected subgroup. Denote by $\Lieg$ and $\Liek$ their respective Lie algebras. Let $\langle \ , \, \rangle$ be an inner product in $\Lieg$, with respect of which $\ad_x^{\top}$ exists. Furthermore, we assume that $\Lieh = \Liek^\perp$ and $\Lieg = \Lieh \oplus \Liek$. Define the horizontal distribution $\calH$ by left translations of $\Lieh$. Let $\mathbf g$ be a left-invariant Riemannian metric on $G$ obtained from $\langle \ , \, \rangle$ and $\mathbf h = \mathbf g|_{\calH}$.  If the metric $\mathbf g$ is invariant under the action of $K$, then it gives us an opportunity to construct normal critical curves from  Riemannian geodesics. 

\begin{theorem}[\cite{GMV}] \label{TheoremLeft}
The following statements hold.
\begin{itemize}
\item[(a)] If $\langle \ , \, \rangle$ is $\ad(\Liek)$ invariant and if $\gamma_R\colon[0,1]\to G$ is a Riemannian geodesic with respect to $\mathbf g$, then
$$\lambda (t)= \pr_{\Liek} \kappa^\ell(\dot \gamma_R(t)),\qquad t\in[0,1]$$
is constant. Here $\pr_{\Liek}\colon\Lieg\to\Liek$ is the orthogonal projection with respect to $\langle \ , \, \rangle$.
\item[(b)] If $\langle \ , \, \rangle$ is $\Ad(K)$ invariant, then a horizontal curve $\gamma_{sR}\colon[0,1]\to G$ is a normal geodesic, if and only if, it is of the form
\begin{equation} \label{eq:sRKinvariant} \nonumber \gamma_{sR}(t) = \gamma_R(t) \cdot \exp_G(-\lambda t), \quad t\in I,
\end{equation}
where $\gamma_R\colon I\to G$ is a Riemannian geodesic with respect to $\mathbf g$, and $\lambda = \pr_{\Liek} \kappa^\ell(\dot \gamma_R(0))$.
\end{itemize}
\end{theorem}

 We emphasize the following fact. 
 \begin{corollary}\label{cor1}
The left logarithmic derivative $u_{sR}$ of a curve $\gamma_{sR}$ satisfies the equation $\dot u_{sR} = \ad_{u_{sR}}^{\top}(u_{sR} + \lambda)$ with a constant $\lambda$.
 \end{corollary}
Notice that the metric $\mathbf g$ does not need to be positively definite on both $\Lieh$ and $\Liek$, it can be positive definite on $\Lieh$ and a pseudometric on $\Liek$ at the same time. Only the transversality of $\Lieh$ and $\Liek$ has to be preserved.
Moreover, Theorem~\ref{TheoremLeft} can be generalized to principal bundles in the case of finite-dimensional manifolds, see~\cite[Theorem 11.8]{Mon}.

\subsection{Controllability on Lie groups}
We defined critical points of energy functional in the set of horizontal curves connecting
 two points $a_1, a_2$. Now we study the problem of controllability,~i.~e., we check if the set of such curves is non-empty.
 A sub-Riemannian structure $(\calH, \mathbf h)$ is called controllable if any two points can be connected by a horizontal curve. The main tool  to prove this property in finite dimensions is the Rashevski{\u\i}-Chow theorem \cite{Chow,Rashevsky}. There are almost no general results on connectivity by horizontal curves in infinite dimensions.

We present here a controllability result for a special class of infinite-dimensional Lie groups. Assume that a horizontal subbundle $\calH$ is invariant under the action of a subgroup $K$ of a given group~$G$. Then, if the tangent bundle $TK$ is transversal to $\calH$, the problem of controllability reduces to the problem whether elements of $K$ can be reached from the unity by a horizontal curve. One of particularly interesting cases is when the subgroup $K$ is finite-dimensional.

\begin{proposition}[\cite{GMV}] \label{groupcontrollability}
Let $G$ be a Lie group with the Lie algebra $\Lieg$, and let a left- (or right-) invariant horizontal subbundle $\calH$ be obtained by left (or right) translations of a subspace $\Lieh \subseteq \Lieg$. Assume that there is a sub-group $K$ of $G$ with the Lie algebra $\Liek$, and such that $\Lieg = \Liep \oplus \Liek$ for some $\Liep \subseteq \Lieh$. Suppose also that $\Lieh$ is $\Ad(K)$-invariant. Then any pair of elements in $G$ can be connected by a smooth horizontal curve, if and only if, for every $a \in K$ there is a horizontal smooth curve connecting $\mathbf 1\in K$ and $a$.
\end{proposition}

\section{The group of diffeomorphisms of $S^1$} \label{sec:Diff}

Let $\Diff S^1$ denote the group of  orientation preserving diffeomorphisms of the unit circle $S^1$, which is the component of the identity of the group of all diffeomorphisms of $S^1$.  
See  \cite{Milnor} for a description of the manifold structure on the diffeomorphism groups. 

We  denote by $\id$ the identity in $\Diff S^1$. Let us identify $T\Diff S^1$ and $\Diff S^1 \times \Vect S^1$    by associating the element $(\gamma(0), \dot \gamma(0) \partial_{\theta})$ to the equivalence class of curves $[t  \mapsto \gamma(t)] \in T_{\gamma(0)} \Diff S^1$ passing through $\gamma(0)$.
The left and right actions are described by
\begin{equation} \label{leftrightDiff} d\ell_{\varphi} (\phi, x \partial_\theta) = \big(\varphi \circ \phi, (\varphi' x) \partial_\theta \big), \qquad
dr_{\varphi} (\phi, x \partial_\theta) = \big(\phi \circ \varphi, (x \circ \varphi) \partial_\theta\big),\end{equation}
where $\phi, \varphi \in \Diff S^1, x \in C^\infty(S^1)$.  Notice that  \eqref{leftrightDiff} implies $\Ad_{\varphi} x \partial_{\theta} = \varphi' x(\varphi^{-1}) \partial_\theta.$

\subsection{Relationship to univalent functions}
Consider the space $\calA_0$ of all holomorphic functions
$$F: \unitD \to \comp, \qquad F(0) = 0,\quad\text{with}\quad \unitD=\{z:\,\,|z|<1\},$$
such that the extension of $F$ to the boundary $S^1$ is $C^\infty(\hat\unitD, \comp)$. Here, $\hat{\unitD}$ denotes the closure of $\unitD$. 
The class $\mathcal A_0$  is a complex Fr\'echet vector space where the topology is defined by the seminorms
$$
\|F\|_m=\sup \{ |F^{(m)}(z)| \ \mid\ z\in\hat{\mathbb D}\},
$$ which is equivalent to the uniform convergence of all derivatives $F^{(m)}$ in $\hat{\mathbb D}$. The local coordinates can be defined by
 the embedding of $\calA_0$ to $\mathbb C^{\mathbb N}$ given by
$$F = \sum_{n=1}^\infty a_n z^n \mapsto (a_1, a_2, \dots ).$$
Let $\calF_0$ be a subclass of $\calA_0$ consisting of all univalent  functions $f\in\calA_0$, normalized by  $f'(0) =1$. The de Branges theorem~\cite{deBranges} yields that $\calF_0$ is contained in the bounded subset
$$1 \times \prod_{n=2}^\infty n \unitD \subseteq \comp^{\mathbb{N}}.$$

Let $\unitD_-$ be the exterior of the unit disk $\mathbb D=\mathbb D_+$. For any $f \in \calF_0$, we define a {\it matching function}  $g:\unitD_- \to \comp$, such that the image of $\mathbb D_-$ under $g$ is exactly the exterior of $f(\mathbb D_+)$, and let $g$ satisfy the normalization $g(\infty) = \infty$. Note that such $g$ exists by the Riemann mapping theorem. Since both functions $f$ and $g$ have a common boundary, $g$ also has a smooth extension to the closure $\hat{\mathbb D}_-$ of $\mathbb D_-$. Therefore, the images $g(S^1)$ and $f(S^1)$ are defined uniquely and represent the same smooth contour in $\comp$. If $g$ and $\widetilde g$ are two matching functions to $f$, then  they are related by a rotation
$$\widetilde g(\zeta) = g(\zeta w),\quad \zeta \in \unitD_-, \quad |w|=1.$$
For an arbitrarily matching function $g$ to $f \in \calF_0$ the diffeomorphism $\phi\in\Diff S^1$, given by
\begin{equation} \label{corsp} 
e^{i\phi(\theta)} = (f^{-1} \circ g)(e^{i\theta}),
\end{equation}
is uniquely defined by $f$ up to the right superposition with a rotation. Let $\Rot$ denote the sub-group of $\Diff S^1$ consisting of rotations. As a Lie group, it is isomorphic to $U(1)$. Its Lie algebra $\rot$ can be identified with the constant vector fields on $S^1$.
The relation \eqref{corsp} gives a holomorphic bijection
\begin{equation} \label{corsp1} \Diff S^1/\Rot   \cong \calF_0,\end{equation}
after complexification of $\Diff S^1/\Rot$, see  \cite{AM, Kirillov0,Kirillov1}. The induced transitive left action of $\Diff S^1$ on $\calF_0$ is holomorphic.

\subsection{Sub-Riemannian structures corresponing to invariant K\"ahler metrics on the univalent functions} \label{sec:sRstructures}
All pseudo-Hermitian metrics on $\calF_0$ which are invariant under the action of $\Diff S^1$ belong to a two-parameter family of metrics $\mathbf b_{\alpha\beta}$, see~\cite{Kirillov1, Kirillov4, Unirreps}. Let us first describe these metrics at $\id_{\unitD} \in \calF_0$.
Any smooth curve $f_t$ in $\calF_0$ with $f_0 = \id_{\unitD}$ is written as
$$f_t(z) = z + t z F(z) + o(t), \qquad F \in \calA_0.$$
Hence, we can identify $T_{\id_{\unitD}} \calF_0$ with $\calA_0$ by relating $[t \mapsto f_t]$ to $F$. With this identification, $\mathbf b_{\alpha\beta}$ is given by
\begin{align}\label{metric}
\mathbf b_{\alpha\beta}\big\vert_{\id_{\unitD}}(F_1, F_2) & = \frac{2}{\pi} \iint_{\unitD} \Big( \alpha F_1' \overline{F}_2' + \beta (z F_1')' \overline{(z F_2')'} \Big) d\sigma(z),\nonumber \\
& = 2\sum_{n=1}^\infty (\alpha n + \beta n^3) a_n \overline{b}_n,
\end{align}
where $d\sigma(z)$ is the area element and
$F_1(z) = \sum_{n=1}^\infty a_n z^n$, $F_2(z) = \sum_{n=1}^\infty b_n z^n$.
This description determines $\mathbf b_{\alpha\beta}$ uniquely, since the metric at any other point $f$ of $\calF_0$ can be obtained by using the left action of $\Diff S^1$, because the metric is invariant. However,  given a univalent function $f$, there is no general method to obtain a matching function $g$, so it is difficult realize the left action of $\Diff S^1$ explicitly (see \cite{GrGumVas} for some concrete examples where matching functions are found).
 
If $\alpha \neq -n^2 \beta$, and $n \in \integer$, then the metric $\mathbf b_{\alpha\beta}$ is non-degenerate pseudo-Hermitian. Otherwise, $\mathbf b_{\alpha\beta}$ is degenerate along a distribution of complex dimension 1. Moreover, we require $\beta \geq 0$ and $- \alpha < \beta$ in order to obtain a positively definite Hermitian metric. If $\alpha=1$ and $\beta=0$, then the metric is called Velling-Kirillov~\cite{Kirillov1, Kirillov4, Vel}.

Since the left action of $\Diff S^1$ on $\calF_0$ is complicated, these metrics can be difficult to study. We lift them to sub-Riemannian metric on $\Diff S^1$, where we have a formula for the left action given by \eqref{leftrightDiff}. Consider the splitting of $\Vect S^1$ into subspaces
$$\Vect S^1 = \Vect_0 S^1 \oplus \rot,$$
where $\Vect_0 S^1$ is the space of all vector fields with vanishing mean value on $S^1$ or in other words $\Vect_0 S^1$ is the kernel of the functional
\begin{equation} \label{eta0} \eta_0(x) = \frac{1}{2\pi} \int_0^{2\pi} x(\theta) d\theta.\end{equation}
Define $\calH$ as the subbundle of $T\Diff S^1$ obtained by left translations of $\Vect_0 S^1$. Horizontal curves with respect to $\calH$ are then curves satisfying
$$\eta_0\left(\kappa^\ell(\dot \gamma(t))\right) =\frac{1}{2\pi} \int_0^{2\pi} \frac{\dot \gamma(t,\theta)}{\gamma'(t,\theta)} \, d\theta = 0 \text{ for any } t \in I.$$
The subbundle $\mathcal H$ is an Ehresmann connection relative to the submersion $\pi: \Diff S^1 \to \calF_0$, since $\calH \oplus \ker d\pi = T\Diff S^1$. The bijective map
$$\begin{array}{rccrc} d_{\id} \pi: & \Vect_0 S^1& \to & T_{\id_{\unitD}}\calF_0 \cong & \calA_0 \\
& x \partial_\theta & \mapsto & & F \end{array}. $$
is given by forumla (see~\cite{Kirillov1}),
$$F(e^{i\theta}) = -\frac{i}{2}\Big(x(\theta) - iJx(\theta)\Big), \qquad Jx(\theta) = \frac{1}{2\pi} \text{p.v.} \int_0^{2\pi} \frac{x(t)}{\tan\left(\frac{t-\theta}{2} \right)} dt.$$
The operator $J$ is the Hilbert transform.

Let us define skew-symmetric bilinear operators on $\Vect S^1$ by the formula
$$\omega_{\alpha\beta}(x,y) = \frac{1}{2\pi}\int_0^{2\pi} \left(\alpha x(\theta) y'(\theta) + \beta x'(\theta) y''(\theta)\right)\, d\theta.$$
Then we have the relation
\begin{align*}
\mathbf b_{\alpha\beta}|_{\id_{\unitD}}\big(d_{\id}\pi x , d_{\id} \pi y \big) = i \omega_{\alpha\beta}(x, y) + \omega_{\alpha\beta}(Jx, y), \qquad x,y \in \Vect_0 S^1,
\end{align*}
see \cite{GMV} for details. The real part of $\mathbf b_{\alpha\beta}$ gives an inner product on $\Vect_0 S^1$. We can extend it to an inner product $(\,,\,)_{\alpha\beta}$ on $\Vect S^1$ by 
$$(x,y)_{\alpha\beta} = \omega_{\alpha\beta}\Big(J(x- \eta_0(x)), y-\eta_0(y)\Big)+ \eta_0(x) \eta_0(y), \qquad x,y \in \Vect S^1.$$
The  inner product on $\Vect_0 S^1$ corresponding to the form $\omega_{\alpha\beta}$ is obtained by
$$(x, y)_{\alpha\beta} = \omega_{\alpha\beta}(Jx,y).$$
This inner product makes $\Vect_0 S^1$ and $\rot$ orthogonal.

Define a Riemannian metric $\mathbf g_{\alpha\beta}$ on $\Diff S^1$ by left translations and use $\mathbf h_{\alpha\beta}$ for its restriction to $\calH$. We want to study the geometry on $\Diff S^1$ with respect to the sub-Riemannian structure $(\calH, \mathbf h_{\alpha\beta})$.

\begin{remark}
If we extend the definition of $J$ to an almost complex structure on $\calH$ by left translation, then $(\Diff S^1, \calH, J)$ becomes an infinite dimensional CR-manifold \cite{Lempert}.
\end{remark}

\subsection{Normal sub-Riemannian geodesics} \label{sec:Normalgeo}
By a result in \cite{GMV}, we know that any two points on $\Diff S^1$ can be connected by a curve that is horizontal to $\calH$. The sub-Riemannian structure $(\calH, \mathbf h_{\alpha\beta})$ is invariant under the action of $\Rot$. Indeed, since $\Rot$ is finite dimensional, we restrict our procedure to Lie algebras. Recall that $\rot$ consists of constant vector fields and  $[\rot, \Vect_0 S^1] \subseteq \Vect_0 S^1$, because the derivative $x'$ has vanishing mean value for any $x \in \Vect S^1$. Since for any $x,y \in \Vect S^1$ we have
$$(x', y)_{\alpha\beta} = -(x, y')_{\alpha\beta},$$
the inner product is invariant under the action of $\ad(\rot)$. Hence, the conditions of Theorem~\ref{TheoremLeft} are satisfied with $K = \Rot$. In order to give a geodesic equation for the normal sub-Riemannian geodesics, we show that the adjoint to $\ad_x$ is well defined with respect to $(\cdot , \cdot)_{\alpha\beta}$. We  consider the inner product
\begin{equation} \label{L2metric} \langle x, y\rangle = \frac{1}{2\pi} \int_0^{2\pi} xy \, d\theta, \quad x,y\in \Vect S^1, \end{equation}
and
notice the relation $(x,y)_{\alpha\beta} = \langle L_{\alpha\beta} Jx' + \eta_0(x), y\rangle$ on $\Vect S^1$, where $L_{\alpha\beta}$ is the second-order differential operator $L_{\alpha\beta} = \beta \partial_\theta^2 - \alpha \cdot$. The adjoint to $\ad_x$ with respect to $\langle \cdot , \cdot \rangle$ is given by the expression
$$\ad_x^\top(y) = x y' + 2x'y.$$
If $\gamma$ is a normal sub-Riemannian geodesic with the left logarithmic derivative $u$, then $u$ is a solution to the equation $L_{\alpha\beta}\frac{d}{dt}Ju'(t)=\ad^{\top}_u(L_{\alpha\beta}Ju'+\lambda)$ for some $\lambda \in \rot \cong \real$, where the adjoint map $\ad^{\top}_u$ is taken with respect to the inner product $\langle\cdot, \cdot\rangle$ in \eqref{L2metric}. Indeed, we have
\begin{align*}
(\dot u, y)_{\alpha\beta} = & \langle L_{\alpha\beta} J \dot u', y \rangle \\
= & ( u+ \lambda, [u,y])_{\alpha\beta} = \langle L_{\alpha\beta} J u' + \lambda, [u,y]\rangle =
\langle \ad_u^\top(L_{\alpha\beta}J u' + \lambda), y\rangle
\end{align*}
for any $y \in \Vect S^1$.
Explicitly, $\gamma$ is a normal geodesic if and only if it is a solution to
\begin{equation} \label{NorKahl}
\kappa^\ell(\dot \gamma) = u,\qquad L_{\alpha\beta} J\dot u' = u L_{\alpha\beta} Ju'' + 2u' L_{\alpha\beta} Ju' + 2 \lambda u',\ \ \lambda\in\mathbb R.\end{equation}

For $(\alpha,\beta) = (1,0)$, this is a special case of the modified  Constantin-Lax-Majda (CLM) equation. For more information, see~\cite{BBHM, EKW}, where the Riemannian geometry with respect to the metric $\mathbf g_{1,0}$  is considered. It can be seen as the Sobolev $H^{1/2}$ metric on $\Diff S^1$.

If we solve equation \eqref{NorKahl} in the special case $\lambda = 0$, and project the solutions to $\calF_0$, we obtain the Riemannian geodesics for $\mathbf b_{\alpha\beta}$. 

\begin{remark}
There are other choices of metrics on $\Diff S^1$ that are interesting from the point of view of PDEs. For example, the geodesic equations with respect the left-invariant metric Riemannian metric induced by $\langle \cdot , \cdot \rangle$ defined as in \eqref{L2metric} is the Burgers' equation $\dot u = 3u u'$. Similarly, Riemannian geodesics with respect to the metric  $\langle x, y \rangle^{1,1} = -\langle L_{1,1} x,y \rangle$ is the non-extended Camassa-Holm equation. The equations for sub-Riemannian geodesics with respect to Sobolev metrics $\langle x, y \rangle^{\alpha,\beta} = -\langle L_{\alpha,\beta} x,y \rangle$, $x,y\in \Vect_0 S^1$, were obtained in~\cite{GMV}.
\end{remark}


\subsection{A sub-Riemannian structure induced by the Weil-Petersson metric}
Let us now consider the group $\Mob$ of M\"obius transforms of the unit disk restricted to the circle~$S^1$. We use the natural embedding of the space $\Diff S^1/\Mob$ to the universal Teichm\"uller space $\mathcal T$ and restrict the Weil-Petersson metric from $\mathcal T$ to $\Diff S^1/\Mob$.
The Lie algebra $\mob$ of $\Mob$ can be considered as a Lie algebra of elements $\lambda \in \Vect S^1$ of the form
\begin{equation} \label{lambdamob} \lambda = \lambda_0 + w e^{i\theta} + \overline w e^{-i\theta}, \qquad \lambda_0 \in \real, \quad w \in \comp.\end{equation}
Let $\eta_0$ be as in \eqref{eta0} and define $\eta_1$ as the $\comp$-valued functional
$$\eta_1(x) = \frac{1}{2\pi} \int_0^{2\pi} x(\theta) e^{-i\theta} \, d\theta, \quad x\in \Vect S^1.$$
We denote the complement to $\mob$ in $\Vect S^1$ by $\Lied = \ker \eta_0 \cap \ker \eta_1$.

Let $\calD$ be the subbundle of $T\Diff S^1$ induced by left translations of $\Lied$. It is an Ehresmann connection with respect to the submersion $\pi\colon  \Diff S^1 \to \Diff S^1/\Mob$. Let us equip $\Diff S^1/\Mob$ with the Hermitian metric by restricting the Weil-Petersson metric from the universal Teuchm\"uller space. We lift it to a sub-Riemannian metric on $\calD$ by the same method described in Section \ref{sec:sRstructures}. This metric will be left-invariant and its restriction to $\Lied$ is given by the inner product
$$( x, y )_{-1,1} = \frac{1}{2\pi}\int_0^{2\pi} (Jx' y'' - Jx y') d\theta = \langle L_{-1,1} Jx', y \rangle \qquad \text{ for any } x,y \in \Lied,$$
see \cite{TT} for details. We extend the inner product $( \cdot ,\cdot )_{-1,1}$ to the whole $\Vect S^1$ by $( \cdot ,\cdot )_{-1,1}+\langle \cdot, \cdot \rangle$, where the metric $\langle \cdot, \cdot \rangle$ is defined on $\mob$, and  $\Lied$ and $\mob$ become  orthogonal with respect to the extended metric. 

By similar arguments as in Section \ref{sec:Normalgeo}, a curve $\gamma$ is a normal sub-Riemannian geodesic, if and only if, $u=\kappa^\ell(\dot \gamma)$ is a solution to the equation
$$L_{-1,1} J\dot u' + \dot\lambda = \ad_u^\top(L_{-1,1} u' + \lambda).$$
for some curve $t \mapsto \lambda(t)$ in $\mob$. However, the sub-Riemannian structutre, both the distribution and the metric, is not invariant under $\Mob$, so we cannot assume  $\lambda$ to be constant.
Solution to the above equation is more complicated because we can not apply Theorem~\ref{TheoremLeft}. Write $\lambda$ as in \eqref{lambdamob}, and define Fourier coefficients $c_n$ of $u$ by
$$u = \sum_{n=2}^\infty(c_n e^{in\theta} + \overline c_n e^{-in\theta}).$$
Observe that 
$$L_{-1,1} Ju' = i \sum_{n=2}^\infty (n^3-n)(c_n e^{in\theta} - \overline c_n e^{-in\theta}).$$
Computing
\begin{align*}
\ad_u^\top(L_{-1,1}J u' + \lambda) = & 2i \lambda_0 \sum_{n=2}^\infty n (c_n e^{in\theta} - \overline c_n e^{-in\theta}) \\
& +3i(\overline w c_2 e^{i\theta} - w \overline c_2 e^{-i\theta}) + 5i(w \overline c_3 e^{i2\theta} - \overline w c_2 e^{-i2\theta})\\
& + i \sum_{n=3}^\infty \left(((2n-1)w c_{n-1} + (2n+1) \overline w c_{n+1})e^{in\theta} \right. \\
& \qquad \left. - ((2n-1)\overline w \overline c_{n-1} + (2n+1) w \overline c_{n+1})e^{-in\theta} \right) \\ 
& + \sum_{n = 4}^\infty \sum_{k=2}^{n-2} i(2n-k) (k^3 -k) (c_k c_{n-k} e^{in \theta} - \overline c_k \overline c_{n-k} e^{-in \theta}) \\ 
& + \sum_{n=2}^\infty \sum_{k= 2}^\infty in (n^2 -1)(2k+n) (\overline c_k c_{k+n} e^{in\theta} - c_{k} \overline c_{k+n} e^{-in\theta}), \\
\end{align*}
we arrive at equations
$$\dot \lambda_0 = 0, \qquad \dot w = 3i\overline w c_2.$$
This means that $u$ must solve the equation
$$\dot v''' + \dot v' = u v'''' + u v'' + 2 u' v''' + 2 u' v' + 2\lambda u' + u\lambda'  - 3i(\overline w c_2 e^{i\theta} - w \overline c_2)e^{-i\theta},$$
where $v=Ju$.

If we solve it for $\lambda = 0$ and $w =0$, and project the solutions to $\Diff S^1/\Mob$, then we get geodesics of the Weil-Petersson metric.

\section{The Virasoro-Bott group} \label{sec:Vir}
Consider the universal cover group $\wDiff S^1$ of $\Diff S^1$, of orientation preserving diffeomorphisms $\phi: \real \to \real$ such that $\phi(\theta + 2\pi) = \phi(\theta) + 2\pi.$ The group $\wDiff S^1$ has a unique non-trivial central extension by $\real$ called the {\it Virasoro-Bott group}.
It can be described as follows. Define a Lie algebra $\Lieg_{\mu\nu}$ as the vector space $\Vect S^1 \oplus \real$, with the commutator
$$\big[(x ,a_1), (y , a_2)\big]  = \Big([x, y ], \omega_{\mu\nu}(x, y )\Big), \quad 
\omega_{\mu\nu}(x,y) = \frac{1}{2\pi}\int_0^{2\pi} \Big(\mu x(\theta) y'(\theta) + \nu x'(\theta) y''(\theta)\Big)\, d\theta.$$
The extension is trivial if and only if $\nu= 0$. All nontrivial extensions with $\nu \neq 0$ are isomorphic. The algebra 2-cocycle $\omega_{\mu\nu}$ is called the Gelfand-Fuchs cocycle. There is a unique simply connected Lie group $\calG_{\mu\nu}$ corresponding to each Lie algebra $\Lieg_{\mu\nu}$, $\calG_{01}=\Vir$. It can be considered as the set $\wDiff S^1 \times \real$ with the group operation
\begin{equation}\label{multVir}
(\phi_1, b_1) (\phi_2,b_2) = \Big(\phi_1 \circ \phi_2, b_1+ b_2 + \mu A(\phi_1, \phi_2) + \nu B(\phi_1,\phi_2)\Big),
\end{equation}
where
$$A(\phi_1, \phi_2) = \frac{1}{4\pi} \int_0^{2\pi} (-\phi_1 \circ \phi_2  + \phi_1 + \phi_2 - \id) d\theta,\quad\id\in\wDiff S^1,$$
$$B(\phi_1, \phi_2) = \frac{1}{4\pi} \int_0^{2\pi} \log (\phi_1 \circ \phi_2)' d \log \phi_2'.$$
The group $\calG_{\mu0}$ is isomorphic to the product group $\wDiff S^1 \times \real$, where the sign $(\times)$  means the direct product of groups, while for $\nu \neq 0$, the extension $\calG_{\mu\nu}$ is non-trivial. All the groups $\calG_{\mu\nu}$ with $\nu \neq 0$ are isomorphic and called the {\it Virasoro-Bott group} because of the Bott cocycle $B(\phi_1, \phi_2)$.

We define a sub-Riemannian structure on $\calG_{\mu\nu}$ in the following way. Define an inner product $\langle \cdot , \cdot \rangle$ on $\Lieg_{\mu\nu}$, by formula
$$\big\langle (x, a_1) , (y,a_2) \big\rangle = \frac{1}{2\pi} \int_0^{2\pi} x(\theta) y(\theta) \, d\theta + a_1 a_2.$$
Notice that with respect to this inner product, the adjoint of $\ad_{(x,a)}$ is given by
$$\ad_{(x,a)}^\top (y,a_0) = (xy'+2x'y +  a_0L_{\mu\nu}x', 0).$$
Consider a splitting $\Lieg_{\mu\nu} = \mathfrak e \oplus \Liek$ given by
$$\mathfrak e = (\Vect_0, 0), \quad \text{and} \quad \Liek = \{ (a_0 \partial_\theta, a) \in \Lieg_{\mu\nu} \, : \, a_0,a \in \real \}.$$
Notice that $\Liek$ is the Lie algebra of the subgroup
$$K = \{(\theta \mapsto \theta + b_0, b) \in \calG_{\mu\nu} \, : \, b_0, b \inÊ\real \},$$
which is an abelian subgroup, isomorphic to $\real^2$. Define a horizontal subbundle $\calE$ of $T\calG_{\mu\nu}$ by left translation of $\mathfrak e$. Similarly, we can define a metric $\widehat{\mathbf h}$ by left translation of the inner product $\langle \cdot , \cdot \rangle$, restricted to $\mathfrak e$.

We claim that the sub-Riemannian structure $(\calE, \widehat{\mathbf h})$ is invariant under the action of $K$. In order to see this, observe that $[\Liek, \mathfrak e] \subseteq \mathfrak e,$ and 
$$0 = \langle [(0, 1), (x,a_1)] , (y,a_2) \rangle = -\langle (x,a_1) , [(0,1),(y,a_2)] \rangle,$$
\begin{align*} & \langle [(1, 0), (x,a_1)] , (y,a_2) \rangle = -\langle (x',0) , (y,a_2) \rangle \\
= & \langle (x,a_1) , (y',0) \rangle = -\langle (x,a_1) , [(1,0),(y,a_2)] \rangle.\end{align*}
In the above equalities the first coordinate in $(0,1)$ and $(1,0)$ mean the constant 0- or 1-function respectively, and the second means just a  number.
Hence, we can apply Corollary \ref{cor1}, and we know that any left logarithmic derivative $u$ of a normal sub-Riemannian geodesic is a solution to $(\dot u, 0) = \ad_{(u,0)}^\top(u+ \lambda_1,\lambda_2), u(t) \in \Vect_0 S^1$, that is,
$$\dot u = 3uu' + 2 \lambda_1 u' + \lambda_2 L_{\nu\mu} u', \qquad \lambda_1, \lambda_2 \in \real.$$
For the special case $(\mu, \nu) = (0,1)$ and with normalization $\lambda_2 =1$, we obtain that $u+\lambda$ is a solution to the KdV-equation.

\begin{remark}
It is also possible to obtain the  Hunter-Saxton and the Camassa-Holm equations as Riemannian geodesic equations on the Virasoro-Bott group. A good overview of these results can be found in \cite{K}.
\end{remark}
\begin{remark}
We do not derive geodesic equations on the sub-Riemannian manifold  $(\Vir, \mathcal C, \mathbf{h}_{\mathcal C})$ defined in Section~1 in this paper, because the procedure is  the same as described in Section 3.4 with a extra KdV-type term.
\end{remark}

\end{document}